\documentclass[12pt]{article}

\usepackage[left=1in,right=1in,top=1in,bottom=1in]{geometry}

\usepackage{subfigure,color}
\usepackage{graphicx}
\usepackage{amsfonts}
\usepackage{amsmath,amssymb}
\usepackage{enumerate}
\usepackage{color}
\usepackage{amsthm}

\newtheorem{example}{Example}

\newtheorem{remark}{Remark}
\newtheorem{lemma}{Lemma}
\newtheorem{theorem}{Theorem}
\newtheorem{corollary}{Corollary}

\newcommand{\R}{{\mathbb R}}
\newcommand{\Z}{{\mathbb Z}}
\newcommand{\eqdef}{\overset{\text{def}}{=}}
\newcommand{\E}{\mathbb E}

\newcommand{\X}{X^N}
\newcommand{\Zm}{Z^N}
\newcommand{\hl}{h_{\ell}}

\newcommand{\Ql}{\widehat{Q}}

\title{Complexity of Multilevel Monte Carlo Tau-Leaping}

\author{
David F. Anderson\thanks{Department of Mathematics, University of
  Wisconsin, Madison, USA.  anderson@math.wisc.edu.},
\and
Desmond J. Higham\thanks{Department of Mathematics and
 Statistics, University of Strathclyde, UK.
d.j.higham@maths.strath.ac.uk.},
  \and
 Yu Sun\thanks{Department of Mathematics, University of
  Wisconsin, Madison, USA.  ysun@math.wisc.edu.}
}

\begin{document}

\maketitle

\begin{abstract}
Tau-leaping is a popular discretization method for generating approximate paths of continuous time, discrete space, Markov chains, notably for biochemical reaction systems.  To compute expected values in this context, an appropriate multilevel Monte Carlo form of tau-leaping has been shown to improve efficiency dramatically.  In this work we derive new analytic results concerning the computational complexity of multilevel Monte Carlo tau-leaping that are significantly sharper than previous ones.   We avoid taking asymptotic limits, and focus on a practical setting where the system size is large enough for many events to take place along a path, so that exact simulation of paths is expensive, making  tau-leaping an attractive option.  We use a  general scaling of the system components that allows for the reaction rate constants and the abundances of species to vary over several orders of magnitude, and we exploit the random time change representation developed by Kurtz.  The key feature of the analysis that allows for the sharper bounds is that when comparing relevant pairs of processes we analyze the variance of their difference directly rather than bounding via the second moment. Use of the second  moment is natural in the setting of a diffusion equation, where multilevel was first developed and where strong convergence results for numerical methods are readily available, but is not optimal for the Poisson-driven jump systems that we consider here. We also present computational results that illustrate the new analysis.
\end{abstract}

\section{Introduction}
\label{sec:int}

Many modeling scenarios give rise to continuous-time, discrete-space Markov chains.
Notable application areas include chemistry, systems biology, epidemiology, population dynamics,
queuing theory and several branches of physics
\cite{Gardiner2002,GHP13,Renshaw2011,Ross2006,Wilkinson2011}.
It is straightforward to simulate sample paths for this class of processes,
but in many realistic contexts the computational cost of performing Monte Carlo is prohibitive.
This work focuses on the commonly arising task of computing an expected value of some
feature of the solution, for example the mean level of a chemical species at some
specified time or the correlation between two population levels.
We study the method proposed in  \cite{AndHigh2012}, which
combined ideas from tau-leaping  and multilevel Monte Carlo  in a
manner that can dramatically improve computational complexity. Our main aim is to provide
further analytical support for the method in the form of substantially sharper bounds for the variances of the coupled processes, and hence for the overall computational complexity of the resulting multilevel Monte Carlo estimator.  The sharper analytic bounds will naturally inform implementation.

Tau-leaping, proposed by Gillespie \cite{Gill2001}, is an Euler-style discretization technique
for generating paths that approximate those of the underlying process.
Given a stepsize, $h$,
the method proceeds by freezing the state of the system over a subinterval of length $h$ and
then updating with an approximate summary of the events that would have taken place.
Intuitively, this approach is attractive when
(a) many events occur over each subinterval, so that exact simulation
would be expensive,  and
(b) the relative change in the system state is small over each subinterval, so that the discretization
is accurate \cite{Anderson2007b,Gill2001,Rathinam2005}.
However, in our case, where the aim is to combine sample paths in order to approximate an expected value,
it should be kept in mind that
 \begin{itemize}
 \item we are concerned with the overall accuracy of the expected value approximation, not
  the accuracy of the individual paths,
 \item forming a Monte Carlo style sample average automatically introduces a statistical sampling error, so we should focus on carefully balancing sampling and discretization effects.
\end{itemize}
In the case of simulating diffusion processes,
these two points were highlighted in the multilevel Monte Carlo approach of Giles \cite{Giles2008},
with related earlier work by Heinrich \cite{Heinrich2001},
 which delivers remarkable improvements in
computational complexity in comparison with standard Monte Carlo.
A key ingredient of multilevel Monte Carlo is to combine paths at different discretization
levels, using relatively fewer paths at the more expensive resolution scales.
The accuracy of the final estimate relies on a recursive control variate construction,
exploiting the fact that pairs of paths
at neighboring resolution levels can be made to have a low variance.

The multilevel philosophy was adapted in \cite{AndHigh2012}  for the case of tau-leaping. This required a novel
simulation approach to generate tightly coupled pairs of Poisson-driven
paths at neighboring resolutions.
The resulting multilevel Monte Carlo tau-leaping algorithm is straightforward to implement and,
unlike the standard multi-level methods utilized in the diffusion case, can be made unbiased by using an exact simulator at the most refined level.  (We note, however, that the recent work of Rhee and Glynn produces an unbiased estimator in the diffusive setting utilizing a randomization idea \cite{RG2012}.)
Computational complexity  was analysed in \cite{AndHigh2012}  for a generic system scaling, and
experimental results confirmed the potential of the method.

 Our aim here is to refine the complexity analysis of \cite{AndHigh2012}.
We do this by directly estimating the variance between relevant pairs of processes,
rather than bounding via the second moment.
Second moment bounds arise naturally in the setting of a
diffusion equation, where multilevel was first developed and where strong convergence results
for numerical methods are readily available, but they are not optimal for the Poisson-driven jump systems
that we consider here.

The next section sets up the notation, defines the simulation methods and discusses some
issues that arise in developing realistic results.
Section~\ref{sec:res}
presents the new analytical results. The consequent bounds on
computational complexity are derived in Section~\ref{sec:complexity}.
Section~\ref{sec:comp} then provides some numerical confirmation of the findings and
Section~\ref{sec:conc} gives conclusions.
Some technical results required  for the analysis are
collected in Appendix~\ref{sec:app}.

\section{Background and Notation}
\label{sec:bn}

 For concreteness, we  follow \cite{AndHigh2012} by using the language of chemical kinetics.  We consider a system with $d$ constituent species, $\{S_1,\dots,S_d\}$, taking part in $K<\infty$ reactions, with the $k$th such reaction written as
\begin{equation*}
	\sum_{i=1}^d \nu_{ki} S_i \to \sum_{i=1}^d \nu_{ki}' S_i.
\end{equation*}
Here,
$\nu_{ki}\in \Z_{\ge 0}$ specifies the number of molecules of type $S_i$ required as input for the $k$th reaction, and $\nu_{ki}' \in \Z_{\ge 0}$ specifies the number of molecules of type $S_i$ produced during the $k$th reaction.  If the $k$th reaction occurs at time $t$, the system is updated via
\begin{equation*}
	X(t) = X(t-) + \nu_{k}' - \nu_k,
\end{equation*}
where $\nu_k$ and $\nu_{k}'$ are  the vectors whose $i$th components are $\nu_{ki}$ and $\nu_{ki}'$, respectively, and $X_i(t)$ denotes the number of $S_i$ molecules at time $t$.
For notational compactness, we let
\[
   \zeta_k \eqdef \nu_k' - \nu_k,
   \]
    which is commonly termed the {\em reaction vector} for the $k$th reaction.   Then  $X(t)$ satisfies
\begin{equation*}
	X(t) = X(0) + \sum_{k=1}^K R_k(t) \zeta_k,
\end{equation*}
where we have enumerated over the reactions, and $R_k(t)$ is the number of times the $k$th reaction has occurred by time $t$.  We typically drop the $K$ from the summation when this will not lead to confusion.

By far the most widely used stochastic model for this system assumes the existence of an
\emph{intensity}, or \emph{propensity},  function $\lambda_k:\R^d \to \R_{\ge 0}$ for the $k$th reaction so that
\begin{equation*}
	R_k(t) = Y_k\left(\int_0^t \lambda_k(X(s)) ds \right),
\end{equation*}
where the collection $\{Y_k\}$ are independent unit-rate Poisson processes.  See, for example, \cite{Anderson2007a,AndKurtz2011,Kurtz72,Kurtz80}. Thus, $X(t)$ is the solution to the stochastic equation
\begin{equation}\label{eq:main}
	X(t) = X(0) + \sum_k Y_k\left( \int_0^t \lambda_k(X(s)) ds\right) \zeta_k.
\end{equation}
The standard choice for the intensity functions $\lambda_k$ is that of mass-action kinetics, which assumes that for $x \in \Z^d_{\ge 0}$
\begin{equation}
	\lambda_k(x) =\kappa_k \prod_{i = 1}^d \frac{x_i!}{(x_i - \nu_{ki})!}1_{\{x_i \ge \nu_{ki}\}}.
	\label{eq:stoch_massAction}
\end{equation}
We note that the continuous-time
Markov chain
(\ref{eq:main}) is equivalent to the model derived by
Gillespie
\cite{Gill76,Gill77}. 
The corresponding forward
Kolmogorov equation
is often referred to as the
Chemical Master Equation \cite{Wilkinson2011}, and solving this very large ODE system forms 
the basis of an 
alternative computational approach 
 that is appropriate when detailed information is required about
the distribution of the solution process
\cite{Jahnke11,MacN2008a}.

In analyzing computational complexity, it is important
to account for the fact that simulation becomes expensive when many reactions take place
along a path, so some sort of ``large system parameter'' is required.
Following
\cite{AndHigh2012,Ball06},
we
denote this parameter by $N$ and
scale the model by setting $X_i^N = N^{-\alpha_i}X_i$, where $\alpha_i \ge 0$ is chosen so that $X^N_i = O(1)$.  The scaled variable then satisfies 
\[
	X_i^N(t) = X_i^N(0) + \sum_k Y_k\left( \int_0^t \lambda_k(X(s)) ds\right) \zeta_{ki}^N,
\]
where $\zeta_{ki}^N = N^{-\alpha_i}\zeta_{ki}$.  
Also following \cite{AndHigh2012,Ball06}, we  set
\begin{equation*}
	\rho_{k} \eqdef  \min\{ \alpha_i\ : \ \zeta^N_{ki} \ne 0 \},
\end{equation*}
so that  $\zeta_{ki}^N = O(N^{-\rho_k})$, with the order achieved for at least one component of $\zeta_k^N$ (i.e. there can be $j$ for which $\zeta_{kj}^N = o(N^{-\rho_k})$).   We also set $\rho \eqdef \min\{\rho_k\} \ge 0$.  

Accounting for the fact that the rate parameters of \eqref{eq:stoch_massAction}, $\kappa_k$, may also vary over several orders of magnitude, for each $k$ there is an $r_k$  for which $\lambda_k(X) = N^{r_k}\lambda_k^N(X^N)$ so that $\lambda_k^N(X^N)$ is $O(1)$.  This yields the stochastic equation
\[
	X^N(t) = X^N(0) + \sum_k Y_k \left(\int_0^t N^{r_k} \lambda_k^N(X^N(s)) ds \right) \zeta_k^N.
\]
 Define now
\[
	\gamma \eqdef \max_k \{r_k-\rho_k\},
\]
which is the natural time-scale for the model.  For $\gamma>0$ the shortest timescale in the problem is much smaller than 1,
and for $\gamma < 0$ it is much larger.  Our
 analysis will be relevant to the
case where $\gamma \le 0$.

Setting $r_k = \gamma + c_k$, we arrive at the stochastic equation for our general scaled model,
\begin{equation}
	X^N(t) = X^N(0) + \sum_k Y_k\left( N^{\gamma} \int_0^t N^{c_k} \lambda_k^N(X^N(s)) ds\right)\zeta_k^N.
	\label{eq:main_multi}
\end{equation}
 Note that, by construction, $c_k \le \rho_k$.
We  may now regard
\begin{equation}
	\overline{N} = N^{\gamma}\sum_k N^{c_k}
\label{eq:Nbar}
\end{equation}
as an order of magnitude for the number of computations required to generate a single path using an exact algorithm.
We refer to
\cite{AndHigh2012,Ball06} for further details and illustrative examples pertaining to the scaling, and
point out that the \emph{classical scaling}
is covered in this framework  by taking
$c_k \equiv \rho_k \equiv 1$ and $\gamma = 0$, see \cite{AndersonGangulyKurtz,Kurtz72}.

Given a stepsize, $h$,
the tau-leaping approximation for the scaled system (\ref{eq:main_multi})
   takes the form
\begin{equation}
	Z_h^N(t) = Z_h^N(0) + \sum_k Y_k \left( N^{\gamma} \int_0^t N^{c_k} \lambda_k^N(Z^N_h ( \eta(s)))
    ds  \right)\zeta_k^N,
\label{eq:RTC_tau_scaled}
\end{equation}
where
$\displaystyle \eta(s) \eqdef \left \lfloor \frac{s}{h} \right \rfloor h$.
Note that we have defined the process over continuous time.

The multilevel tau-leaping method for approximating
$\E[ f(X^N(T))]$
uses
levels $\ell  = \ell_0,\ell_0 + 1,\dots,L$, where both $\ell_0$ and  $L$ are to be determined.
The characteristic stepsize at level $\ell$ is given by $\hl = T \cdot M^{-\ell}$, for a fixed positive integer $M$, typically between 2 and 7.
Using $Z^N_{\ell}$ to denote the tau-leaping process \eqref{eq:RTC_tau_scaled} generated with a step-size of  $h_{\ell}$, we consider the telescoping sum identity
\begin{align}\label{eq:add_sub_real}
  \E [f(Z^N_{L}(T))] = \E [ f(Z^N_{\ell_0}(T)) ] + \sum_{\ell = \ell_0+1}^L \E[ f(Z^N_{\ell}(T)) - f(Z^N_{\ell - 1}(T))].
\end{align}
Exploiting the right-hand side of this identity, we introduce the estimators
\begin{align}
	\Ql_{\ell_0} &\eqdef \frac{1}{n_{0}} \sum_{i = 1}^{n_{0}} f(Z^N_{\ell_0,[i]}(T)), \quad \text{and} \quad
	\Ql_{\ell} \eqdef \frac{1}{n_{\ell}} \sum_{i = 1}^{n_{\ell}} ( f(Z^N_{\ell,[i]}(T)) - f(Z^N_{\ell - 1,[i]}(T))),
	\label{eq:QL_real}
\end{align}
where the pairs $Z^N_{\ell,[i]}$ and $Z^N_{\ell - 1,[i]}$ are to be generated in such a way that $\mathsf{Var}(\Ql_{\ell})$ is small.    We will then let
\begin{equation}
  \Ql_{B} \ \eqdef \   \sum_{\ell = \ell_0}^L \Ql_{\ell},
  \label{eq:estimator_real}
\end{equation}
be the unbiased estimator for $\E[f(Z^N_{L}(T))]$.  Note that $\Ql_{B}$ is therefore a biased estimator for the
 quantity $\E [f(X^N(T))]$.

We emphasize at this stage that the number of levels, $L - \ell_0 + 1$,
and the number of paths per level,
$n_{0}$ and $n_{\ell}$, have not yet been specified, and will be determined by the
required accuracy.

In a similar manner, we may exploit the fact that exact simulation is possible
and use the identity
\begin{equation*}
	\E[ f(X^N(T))] = \E [f(X^N(T)) - f(Z^N_L(T))] + \sum_{\ell = \ell_0+1}^L \E [ f(Z^N_{\ell}) - f(Z^N_{\ell-1})] + \E [f(Z^N_{\ell_0}(T))].
\end{equation*}
We then define estimators for the three types of terms on the right-hand side via
(\ref{eq:QL_real}) and
\begin{equation}
\Ql_{E} \eqdef \frac{1}{n_{E}} \sum_{i = 1}^{n_{E}} (f(X^N_{[i]}(T) - f(Z^N_{L,[i]}(T))) ,
\label{ub3}
\end{equation}
leading to
\begin{equation}
	\Ql_{UB}  \eqdef  \Ql_E +  \sum_{\ell = \ell_0}^L \Ql_{\ell},
	\label{eq:unbiased_MLMC}
\end{equation}
which
is  an \emph{unbiased} estimator for $\E [f(X^N(T))]$.

In the multilevel framework, rather than single paths, we generally compute pairs of paths, and tight
coupling is the key to success \cite{AK2014}.  Defining $a\wedge b \eqdef \min\{a,b\}$, the pair $(Z^N_{\ell},Z^N_{\ell-1})$  appearing in
(\ref{eq:QL_real})
is defined as the solution to the stochastic equation
\begin{align}
\begin{split}
	\Zm_{\ell}(t)  &= \Zm_{\ell}(0) + \sum_{k} \zeta_k^N \bigg[Y_{k,1}\left(N^{\gamma}N^{c_k}  \int_0^t \lambda_k^N(\Zm_{\ell} ( \eta_{\ell}(s))) \wedge \lambda_k^N(\Zm_{\ell-1} ( \eta_{\ell-1}(s))) ds \right) \\
	& +  Y_{k,2}\left(N^{\gamma}N^{c_k} \int_0^t [ \lambda_k^N(\Zm_{\ell}( \eta_{\ell}(s))) -  \lambda_k(\Zm_{\ell} ( \eta_{\ell}(s))) \wedge \lambda_k(\Zm_{\ell-1} ( \eta_{\ell-1}(s)))] ds \right)\bigg],  \label{eq:Z1}
    \end{split}\\
    \begin{split}
	\Zm_{\ell-1}(t) &= \Zm_{\ell-1}(0) + \sum_{k} \zeta_k^N \bigg[Y_{k,1}\left( N^{\gamma}N^{c_k}  \int_0^t \lambda_k(\Zm_{\ell} ( \eta_{\ell}(s))) \wedge \lambda_k(\Zm_{\ell-1} ( \eta_{\ell-1}(s)))  ds  \right)\\
	& +  Y_{k,3}\left( N^{\gamma}N^{c_k} \int_0^t  [\lambda_k(\Zm_{\ell-1}( \eta_{\ell-1}(s))) -  \lambda_k(\Zm_{\ell} ( \eta_{\ell}(s))) \wedge \lambda_k(\Zm_{\ell-1} ( \eta_{\ell-1}(s))) ]ds \right)\bigg],
    \label{eq:Z2}
    \end{split}	
\end{align}
 where the $Y_{k,i},\ i \in \{1,2,3\}$, are independent,
unit-rate Poisson processes, and for each $\ell$,
we define
$\eta_{\ell}(s) \eqdef \lfloor s/h_{\ell}\rfloor h_{\ell}$.

Similarly, for
(\ref{ub3}),
    the pair $(X^N,Z^N_{L})$ is defined as the solution to the stochastic equation
\begin{align}
\begin{split}
	\X(t) = &X^N(0) + \sum_{k} Y_{k,1}\left( N^{\gamma}N^{c_k} \int_0^t  \lambda_k^N(\X(s))  \wedge  \lambda_k^N(\Zm_{L}( \eta_{L}(s)))  ds  \right)\zeta_k^N\\
	&\hspace{.1in} + \sum_{k} Y_{k,2}\left( N^{\gamma}N^{c_k} \int_0^t [ \lambda_k^N(\X(s))  -  \lambda_k^N(\X(s))  \wedge   \lambda_k^N(\Zm_{L}( \eta_{L}(s))) ] ds \right)\zeta_k^N,
    \label{eq:Z_X1}
\end{split}\\
\begin{split}
	\Zm_{L}(t) = &\Zm_{L}(0) + \sum_{k} Y_{k,1}\left(N^{\gamma}N^{c_k}  \int_0^t \lambda_k^N(\X(s))  \wedge   \lambda_k^N(\Zm_{L}( \eta_{L}(s)))  ds \right)\zeta_k^N \\
	&\hspace{.1in} + \sum_{k} Y_{k,3}\left(N^{\gamma}N^{c_k}  \int_0^t  [\lambda_k^N(\Zm_{L}( \eta_{L}(s)))  - \lambda_k^N(\X(s))  \wedge  \lambda_k^N(\Zm_{L}( \eta_{L}(s)))] ds  \right)\zeta_k^N.
      \label{eq:Z_X2}
    \end{split}	
\end{align}

In practice, both the biased and the unbiased multilevel Monte Carlo methods
described above can be implemented straightforwardly;
full pseudo-code descriptions are given in
 \cite{AndHigh2012}.
In this setting, the exact
paths $X^N$ are essentially being simulated by
the \emph{next reaction method}
\cite{Gibson2000}, which has a natural
relation to the random change of time representation
(\ref{eq:main}); see \cite{Anderson2007a}.

To analyze the computational complexity of these methods, we assume that
$\E[f(X^N(T))]$ is required to an accuracy of $\epsilon$, in the sense of a
confidence interval.
Hence, we have in mind that
$\epsilon$ is small.
As discussed earlier, we also wish to regard the system
parameter, $N$, as large. It is important, however, to keep in mind that
we have a `moving target.'
As $N$ increases the properties of the underlying system may vary.
In particular, for the classical scaling, in the
 \emph{thermodynamic limit}
 $N \to \infty$ the stochastic fluctuations become negligible
\cite{AndKurtz2011,Kurtz72} and the system approaches a deterministic
ODE.
We are implicitly assuming that the problem is specified in a regime where fluctuations are of interest
and the task is computationally challenging, so we regard
$\epsilon$ as small and
$N$ as large without committing ourselves to asymptotic limits.
A key aim in our analysis is to capture the effect that the
prescribed task can become less costly as $N$ increases, in the sense that fluctuations decay.

The level-wise estimators
$ \Ql_{\ell}$
and
$\Ql_E$
appearing in
(\ref{eq:estimator_real})
and
(\ref{eq:unbiased_MLMC})
are independent, and hence their variances add.
To make the overall variance achieve the desired value of
$O(\epsilon^2)$, our aim is to choose the number of paths,
$n_{\ell}$,
so that each level contributes equally to the overall variance.
The goal, in terms of obtaining a good upper bound on the computational complexity of the
method, is therefore to develop tight bounds on
the variances of
$ f(Z^N_{\ell,[i]}(T)) - f(Z^N_{\ell - 1,[i]}(T)) $
and
$f(X^N_{[i]}(T) - f(Z^N_{L,[i]}(T))$.

We note that multilevel was first developed and analyzed for diffusion processes
\cite{Giles2008}. Here the
classical and well-studied concept of
$L^2$ strong error of a numerical method
can be used.
Two processes that are both close to the underlying exact process in  the $L^2$ norm must also be close
to each other, and the second moment trivially bounds the variance.
This approach has proved useful
in connection
with
Brownian motion,
\cite{Giles2007,GHM2009,HMSY2013},
and is optimal in the sense that complexity results can be derived that
match the residual $O(\epsilon^{-2})$ cost that would remain if a simulable expression for the
 exact solution were known.
The $L^2$ approach was also used in
 \cite{AndHigh2012} for the Possion-driven case that we consider here, but our aim is to
show that improved bounds can be obtained from a more refined analysis
that studies the variances directly.  

\section{Analyzing the Variance of the Coupled Processes}
\label{sec:res}

We begin by explicitly detailing the \textit{running assumption} we make on the model. 
 We will suppose that for each $k$, the intensity function for the normalized process $ \lambda_k^N(x)$, along with the components of its gradient and Hessian, are bounded.

\vspace{.1in}

\noindent \textbf{Running assumption.}  \textit{We suppose there is a constant $C>0$ such that,
    \[
     \sum_{k}|\lambda_k^N(x)|\leq C,\quad\sum_{k}\left|\frac{\partial \lambda_k^N(x)}{\partial x_j}\right|\leq C, \ \  and \ \ \sum_{k}\left|\frac{\partial^2 \lambda_k^N (x)}{\partial x_i\partial x_l}\right|\leq C\quad for \,\,any\,\,i,j,l = 1,...,d.
\]
 Hence, letting
 \begin{equation*}
	F^N(x) \eqdef \sum_{k} N^{\gamma}N^{c_k}\lambda_k^N(x)\zeta_k^N,
\end{equation*}
 we have
\[
\left|\frac{\partial F^N_i}{\partial x_j}\right|\leq CN^{\gamma} \quad and \quad \left|\frac{\partial^2 F^N_i}{\partial x_k\partial x_l}\right|\leq CN^{\gamma}\quad for \,\,any\,\,i,j,k,l=1,2,...,d.
\]
We will also assume throughout that our initial condition is fixed.  That is,  $X^N(0) \equiv x(0) \in \R_{\ge 0}$ for all choices of $N$.} 
\hfill $\square$

We note that not all reaction networks satisfy the above running assumption.  However, any model for which there is a $w \in \R^d_{>0}$ with $w \cdot \zeta_k \le 0$ for all $k$ will satisfy the running assumption  so long as $\lambda_k^N$ is  defined to satisfy it outside the positive orthant. (Note that the $\tau$-leap process may leave the positive orthant.)  In particular, any process which satisfies a conservation relation will satisfy the running assumption.  We further wish to emphasize that the boundedness assumption is being made on the \textit{scaled} intensity function, and not the intensity function for the unscaled process.  Since the scaled intensity function is explicitly constructed to satisfy $\lambda_k^N(x) =O(1)$ \textit{in our region of interest},  it is not an unrealistic assumption.

\begin{example}
\label{example:running1}

	Consider the family of models, indexed by $N$, with  the following network, rate parameters, which are placed alongside the reaction arrows, and initial condition,
\[
	2A \overset{1/N}{\underset{1}{\rightleftarrows}} B, \qquad X_1(0) = X_2(0) = M\cdot N,
\]
where $X_1(t)$ and $X_2(t)$ denote the abundances of the $A$ and $B$ molecules, 
respectively, at time $t$,  $M > 0$ is some constant, and only those $N$ for which $M\cdot N \in \Z$ are considered.  Note that $w = [1, 2]^T$ satisfies $w \cdot \zeta_k = 0$ for each $k \in \{1,2\}$.
Choosing $\alpha_i \equiv 1$, so that 
$\rho_k \equiv r_k \equiv c_k \equiv \rho = 1$, and $\gamma = 0 $,
the normalized process satisfies
\begin{align}
\begin{split}
	X^N(t) = X^N(0) + \frac{1}{N} &Y_1\left( N \int_0^t \lambda_1^N(X^N(s)) ds\right)\left[\begin{array}{c}
	-2\\
	1
	\end{array}\right] \\
	 + \frac{1}{N} &Y_2\left( N  \int_0^t \lambda_2^N(X^N(s)) ds \right) \left[\begin{array}{c}
	2\\
	-1
	\end{array}\right],
	\end{split}
	\label{eq:coupled_example}
\end{align}
with
\begin{align*}
	\lambda_1^N(x) &= \left\{\begin{array}{cc}
		x_1(x_1-N^{-1}), & \text{ if } x_1,x_2 \ge  0\\[1ex]
		h_1^N(x), &  \text{ if } x_1 < 0\text{ or } x_2 <0
		\end{array}\right. ,
		\end{align*}
		where $h_1^N(x)$ is any function which ensures $\lambda_1^N$ is differentiable at the boundary of the positive orthant and satisfies our running assumptions
and
\begin{align*}
		 \lambda_2^N(x) &=  \left\{\begin{array}{cc}
		x_2, & \text{ if } x_1,x_2 \ge 0 \\[1ex]
		h_2^N(x), &  \text{ if } x_1 < 0\text{ or } x_2 <0
		\end{array}\right.  ,
\end{align*}
where $h_2^N(x)$ is also chosen to ensure the satisfaction of our running assumptions.   For example, we only require that $h_2^N(x) = 0$ when $x_2 = 0$, $h_2^N(x) = (3/2)\lfloor M\cdot N\rfloor$ when $x_1 = 0$, and $h_2^N(x)$ itself satisfies the running assumptions outside the positive orthant.
Note that it is necessary to define each $\lambda_k^N$ outside $\Z^2_{\ge 0}$ for the tau-leaping process, which does not satisfy the same constraints as the exact process.
Also note that a bound of $9M^2$ can be used to satisfy our running assumption for this (non-linear) model.

\end{example}

 \subsection{Statements and proofs of our main results}

We couple $X^N$ and $Z^N_h$ via \eqref{eq:Z_X1} and \eqref{eq:Z_X2} and are interested in the problem of estimating
\[
	\textsf{Var}\left( f(X^N) - f(Z^N_h)\right),
\]
for Lipschitz functions $f:\R^d \to \R$.  Our main result is the following.
	\begin{theorem}
  \label{thm:var}
		Suppose the model satisfies our running assumption and that $X^N$ and $Z_h^N$ satisfy the coupling \eqref{eq:Z_X1} and \eqref{eq:Z_X2}.  Assume that $f:\R^d \to \R$ has continuous second derivative and there exists a constant $M$ such that
     \[
       \left|\frac{\partial f}{\partial x_i}\right|\leq M \quad and \quad \left|\frac{\partial^2 f}{\partial x_i\partial x_j}\right|\leq M\quad for \,\,any\,\,i,j=1,2,...,d.
     \]
		Then, for $0 \le t \le T$, 
		\[
			\textsf{Var}(f(X^N(t)) - f(Z_h^N(t)))\leq \bar{C_1}(N^{\gamma}T,d,M)N^{-\rho}(N^{\gamma}h)^2+\bar{C_2}(N^{\gamma}T,d,M)N^{-\rho}(N^{\gamma}h),
		\]
where $\bar{C_1}$ is defined as (\ref{eq:C1}) and $\bar{C_2}$ is defined as (\ref{eq:C2}).    
\end{theorem}
Note that  the functional dependence of  $\bar C_1$ and $\bar C_2$ on $N$ and $T$ is via the product $N^\gamma T$, so 
 in the case of $T$ fixed and $\gamma\leq 0$, the values of $\bar{C_1}$ and $\bar{C_2}$ may be bounded above uniformly in $N$. 

\begin{remark}
	If it is the case that $X$ remains in a bounded region with a probability of one, for example if there is a $w \in \R^d_{>0}$ with $w \cdot \zeta_k \le 0$ for all $k$, then the assumption pertaining to the derivative bounds of $f$ is automatically satisfied for any smooth $f$.
\end{remark}


The main results in \cite{AndHigh2012}, which used  the $L^2$ norm of the difference of the coupled processes,  present an upper bound for  the variance of $f(X^N(t)) - f(Z_h^N(t))$ of the form  $\bar{C}_{11}(N^\gamma h)^2 + \bar{C}_{22}N^{-\rho} (N^\gamma h)$ for  constants $\bar{C}_{11}$ and $\bar{C}_{22}$ which are similar to $\bar{C}_{1}$ and $\bar{C}_{2}$ of Theorem \ref{thm:var}.  Hence, a key improvement in the present work lies in the multiplication by $N^{-\rho}$ of the term $(N^\gamma h)^2$.   Importantly, if $\gamma \le 0$ and if $\bar{C_1}$ is not significantly larger than $\bar{C_2}$, we may now conclude that the variance has an upper bound that is dominated by a term of the form $N^{-\rho} (N^\gamma h)$ multiplied by a constant.  This is a conclusion that the analysis in \cite{AndHigh2012} does not allow us to reach under similarly reasonable assumptions.



An immediate corollary to Theorem \ref{thm:var} is that the coupled tau-leaping processes satisfy a similar bound.

	\begin{corollary}
  \label{thm:cor}
 Under the assumptions of Theorem~\ref{thm:var},
 $\Zm_{\ell}$ and
$\Zm_{\ell-1}$
in (\ref{eq:Z1}) and
(\ref{eq:Z2})
satisfy, for $0 \le t \le T$, 
		\[
			\textsf{Var}(f(\Zm_{\ell}(t)) - f(\Zm_{\ell-1}(t)))\leq \bar{D_1}(N^{\gamma}T,d,M)N^{-\rho}(N^{\gamma}h_\ell)^2+\bar{D_2}(N^{\gamma}T,d,M)N^{-\rho}(N^{\gamma}h_\ell).
		\]
	\end{corollary}
	
	\begin{proof}
		This is an immediate consequence of  the following simple inequality
		\begin{align*}
			\textsf{Var}&(f(\Zm_{\ell}(t)) - f(\Zm_{\ell-1}(t)))=\textsf{Var}(f(\Zm_{\ell}(t)) - f(X^N(t)) + f(X^N(t)) -  f(\Zm_{\ell-1}(t)))\\[1ex]
			&\le 2\textsf{Var}(f(\Zm_{\ell}(t)) - f(X^N(t))) + 2\textsf{Var}(f(X^N(t)) -  f(\Zm_{\ell-1}(t))),
		\end{align*}
		and Theorem \ref{thm:var}.
	\end{proof}

In order to prove Theorem \ref{thm:var}, we first present some preliminary calculations on the variance of the processes $X^N$ and $Z^N_h$, and some useful lemmas.

        By our running assumption, we know that
        \[
          |F_i^N(x)-F_i^N(y)|\leq \sqrt{d}CN^{\gamma}|x-y|,
        \]
        for all $x,y$ and $i=1,2,...,d$, where $|\cdot |$ is the 2-norm.
We let $x^N$ be the solution to the \textit{deterministic} equation
\begin{equation}\label{gen_approx_odex}
	x^N(t) = x(0) + \int_0^t F^N(x^N(s))ds,
\end{equation}
where $x(0)$ is the initial condition for each  member of our family of processes.
\begin{lemma}
\label{lem:main_var_bound_X}
	Under our running assumption, for any $T>0$ we have
\[
			\E\left[ \sup_{s \le T} |X^N(s) - x^N(s)|^2\right] \leq   \left(C_1 e^{C_2\cdot (TN^{\gamma})^2}\right)\left(TN^{\gamma} N^{-\rho}\right), 
\]
where $C_1$ and $C_2$ are two positive constants that do not depend on $N,\gamma$, $T$.
\end{lemma}
\begin{proof}
Defining $\tilde Y_k(u) := Y(u) - u$ to be the centered Poisson process we have
		\begin{align*}
			X^N(t) - x^N(t) &= \sum_k \tilde{Y}_{k}\left(N^{\gamma}N^{c_k}\int_0^t\lambda_k^N(X^N(s))ds\right)\zeta_k^N+ \int_0^t F^N(X^N(s)) - F^N(x^N(s)) ds,
		\end{align*}
and so using the trivial bound $(a + b)^2 \le 2a^2 + 2b^2$ plus our running assumption,
\begin{align*}
 |X^N(t) &- x^N(t)|^2\\
 & \leq 2 \left| \sum_k\tilde{Y}_{k}\left(N^{\gamma}N^{c_k}\int_0^t\lambda_k^N(X^N(s))ds\right)\zeta_k^N \right|^2 + 2 C^2 td N^{2\gamma}\int_0^t |X^N(s) - x^N(s)|^2 ds.
 \end{align*}
 Hence, by the Burkholder-Davis-Gundy inequality \cite{BDG1972} and our running assumption
 \begin{align*}
\E& \left[\sup_{s\le t} |X^N(s) - x^N(s)|^2\right] \\
&\le 2 N^{\gamma}\sum_k |\zeta^N_{k}|^2 \int_0^t N^{c_k}\E   \left[ \lambda_k^N(X^N(r)) \right] dr +2C^2 t d  N^{2\gamma}\int_0^t \E \left[ \sup_{s \le r} |X^N(s) - x^N(s)|^2\right] dr\\
& \le C_1N^{\gamma}N^{-\rho}t+ 2C^2 t d N^{2\gamma}\int_0^t \E \left[ \sup_{s \le r} |X^N(s) - x^N(s)|^2\right] dr,
\end{align*}
and the result follows from Gronwall's inequality with $C_2 = 2C^2d$.
\end{proof}

Now we let $z$ be the deterministic solution to
\begin{equation}\label{gen_approx_ode}
	z^N_h(t) = x(0) + \int_0^t F^N(z^N_h(\eta(s)))ds,
\end{equation}
which is the Euler approximate solution to the ODE \eqref{gen_approx_odex}.

\begin{lemma}\label{lem:main_var_bound_Z}
	Under our running assumptions, for any $T>0$ 
	\[
		 \E \left[\sup_{0\le s \le T}	 |Z_h^N(s) - z^N_h(s)|^2\right] \le 
                         \left( C_1 e^{C_2 \cdot (TN^\gamma)^2}  \right) (TN^{\gamma} N^{-\rho}),
\]
where $C_1$ and $C_2$  are the same constants which appear in Lemma \ref{lem:main_var_bound_X}. 
\end{lemma}

\begin{proof}
Following the proof of Lemma \ref{lem:main_var_bound_X}, we have
\begin{align*}
			Z^N_h(t) - z^N_h(t) &= \sum_k \tilde{Y}_{k}\left(N^{\gamma}N^{c_k}\int_0^t\lambda_k^N(Z^N_h(\eta(s)))ds\right)\zeta_k^N\\
			&\hspace{.2in}+ \int_0^t F^N(Z^N_h(\eta(s))) - F^N(z^N_h(\eta(s))) ds,
		\end{align*}
and, again as a result of the Burkholder-Davis-Gundy inequality \cite{BDG1972} and our running assumptions,
\begin{align}
\E &\left[\sup_{s\le t}  \left|Z^N_h(s) - z^N_h(s)\right|^2\right] \notag\\
& \le 2 N^{\gamma}\sum_k |\zeta^N_{k}|^2 \int_0^t N^{c_k}\E \left[  \lambda_k^N(Z^N_h(\eta(s))) \right] ds+2C^2 t d N^{2\gamma}\int_0^t\E\left[ \sup_{s\le r} |Z^N_h(\eta(s)) - z^N_h(\eta(s))|^2\right] dr\notag\\
& \le C_1N^{\gamma}N^{-\rho}t+C_2tN^{2\gamma}\int_0^t \E  \left[\sup_{s\le r} |Z^N_h(s) - z^N_h(s)|^2\right] dr,\label{eq:needed_bound14545}
\end{align}
where $C_1$ and $C_2$  are the same constants which appear in Lemma \ref{lem:main_var_bound_X}. An application of Gronwall's inequality now completes the proof.
%
%
\end{proof}

Our proof of Theorem \ref{thm:var} will begin by considering the difference between $X^N$ and $Z_h^N$.  It is therefore useful to get bounds on the second moment of the quadratic variation of a certain martingale which will arise naturally.  We therefore let 
\begin{align}
\begin{split}
	M^N(t) \eqdef  \ &\sum_{k}   \tilde Y_{k,2}\left(N^{\gamma}N^{c_k}\int_0^t [ \lambda_k^N(X^N(s))-\lambda_k^N(X^N(s))\wedge \lambda_k^N(Z^N_h(\eta(s)))]ds\right)\zeta^N_k\\
         &-\sum_{k} \tilde Y_{k,3}\left(N^{\gamma}N^{c_k}\int_0^t [\lambda_k^N(Z^N_h(\eta(s)))-\lambda_k^N(X^N(s))\wedge \lambda_k^N(Z^N_h(\eta(s))) ] ds\right)\zeta^N_k\\
	\end{split}
	\label{eq:M_f}
\end{align}
where, again, $\tilde Y_k$ is the centered Poisson process.
The proof of the following lemma can be found in \cite{AndHigh2012}.

 \begin{lemma}\label{lem:martingale_bound} Under our running assumption,
	\[
		 \E [|M^N(t)|^2] \leq c_1Te^{c_2N^{\gamma}T}N^{2\gamma}(N^{-\rho}h),
	\]
where $c_1,c_2$ are independent of $N,\gamma$ and $T$.
 \end{lemma}

	We are now ready to prove our main result.
	
\begin{proof}[Proof of Theorem \ref{thm:var}]
		We first prove the result in the case that $f_i(x) = x_i$.  We  have
		\begin{align}
			X_i^N(t) &- Z_{h,i}^N(t) = M_i^N(t) + \int_0^t F_i^N(X^N(s)) - F_i^N(Z_h^N(\eta(s))) ds\notag \\
			&= M_i^N(t) + \int_0^t F_i^N(X^N(s)) - F_i^N(Z_h^N(s)) ds + \int_0^t F_i^N(Z_h^N(s)) - F_i^N(Z_h^N(\eta(s))) ds,
			\label{eq:3terms}
		\end{align}
		where $M^N(t)$ is defined in \eqref{eq:M_f}.
%

		We will prove the result by first considering  the variance of the first and third terms of \eqref{eq:3terms}.  Consideration of the second term will then lead naturally to an application of Gronwall's inequality.
		
		To begin, we note that Lemma~\ref{lem:martingale_bound} implies that
		\[
			\textsf{Var}(M^N_i(t)) \leq c_1Te^{c_2N^{\gamma}T}N^{2\gamma}(N^{-\rho}h),
		\]
		for some $c_1$ and $c_2$ which are positive and do not depend on $N^{\gamma}$ and $T$.

		Turning to the third term of \eqref{eq:3terms}, we have the following lemma.
		
\begin{lemma}\label{lemma:third_term}
	\begin{align*}&\displaystyle \textsf{Var}\left( \int_0^t F_i^N(Z_h^N(s)) - F_i^N(Z_h^N(\eta(s)))ds \right)\\
 &\hspace{.5in}\le d\hat{C}_1 \cdot (TN^{\gamma})^2 N^{-\rho}(N^{\gamma}h)^2+d \hat{C}_2 \cdot (T N^{\gamma})^2N^{-\rho}N^{\gamma}h, 
            \end{align*}
 where $\hat{C}_1$ is a constant defined via \eqref{eq:bc}, and which  depends upon the product $T N^\gamma$, and  $\hat{C}_2$ is a constant independent of $N,\gamma$ and $T$.
 \end{lemma}
		
		\begin{proof}
      From Lemma \ref{lem:taylor} in the appendix, we have
       \begin{align}
       \begin{split}
       F_i^N(Z_h^N(s)) &- F_i^N(Z_h^N(\eta(s)) \\
       &=\int_0^1 \nabla F_i^N(Z_h^N(\eta(s)) + r(Z_h^N(s)-Z_h^N(\eta(s))))dr\, \cdot\, (Z_h^N(s)-Z_h^N(\eta(s))).
       \end{split}
       \label{eq:crazy23423}
       \end{align}
      In order to bound the right hand side of \eqref{eq:crazy23423},  we will apply Lemma \ref{lem:Var_bound} in the appendix with $A^{N,h}$ as the $j$th component of $\int_0^1 \nabla F_i^N(Z_h^N(\eta(s)) + r(Z_h^N(s)-Z_h^N(\eta(s))))dr$ and $B^{N,h}$ as the $j$th component of $(Z_h^N(s)-Z_h^N(\eta(s)))$.  Hence, we must find appropriate bounds on these components.

      We begin with  $Z_h^N(s)-Z_h^N(\eta(s))$.  As
       \begin{equation}
	Z_h^N(s) = Z^N(0) + \sum_kY_{k}\left(N^{\gamma}N^{c_k}\int_0^t \lambda_k^N(Z^N_h(\eta(s)))ds\right)\zeta_k^N,
       \end{equation}
       we see
       \begin{align*}
	     Z_h^N(s)-Z_h^N(\eta(s)) &= \sum_k Y_{k}\left(N^{\gamma}N^{c_k}\int_0^s \lambda_k^N(Z^N_h(\eta(s)))ds\right)\zeta_k^N\\
	     &\hspace{.4in}- \sum_kY_{k}\left(N^{\gamma}N^{c_k}\int_0^{\eta(s)} \lambda_k^N(Z^N_h(\eta(s)))ds\right)\zeta_k^N\\
         &\overset{d}{=}\sum_k{Y_k}'\left(N^{\gamma}N^{c_k}\int^s_{\eta(s)} \lambda_k^N(Z^N_h(\eta(s)))ds\right)\zeta_k^N, 
       \end{align*}
       where the collection $\{{Y_k}'\}$ are independent unit-rate Poisson processes and the last equality is in the sense of distributions. This implies
       \begin{align*}
    	   \left|\E [ Z^N_h(s)-Z^N_h(\eta(s))]\right| &= \left| \E \left[ \E [ Z^N_h(s)-Z^N_h(\eta(s))| Z^N_h(\eta(s)) ]\right] \right| \leq CN^{\gamma}h.
       \end{align*}
             Similarly, using Lemma \ref{lem:main_var_bound_Z} and  the law of total variance, we may conclude
         \begin{align}
       \label{eq:24089572}
       \begin{split}
       \textsf{Var}&(Z^N_{h,j}(s)-Z^N_{h,j}(\eta(s))) =\E\left[\textsf{Var}(Z^N_{h,j}(s) - Z^N_{h,j}(\eta(s)))\ |\ Z^N_h(\eta(s)))\right]\\
       &\hspace{.9in} +\textsf{Var}\left(\E[Z^N_{h,j}(s)- Z^N_{h,j}(\eta(s))]\ |\ Z^N_h(\eta(s))\right)\\
       &=\E\left[\textsf{Var}\left(\sum_k\tilde{Y}_{k}\left(N^{\gamma}N^{c_k}\int^s_{\eta(s)} \lambda_k^N(Z^N_h(\eta(s)))ds\right)\zeta_{kj}^N\ \bigg|\ Z^N_h(\eta(s))\right)\right]\\
       &\hspace{.2in}+\textsf{Var}\left(\E\left[\sum_k\tilde{Y}_{k}\left(N^{\gamma}N^{c_k}\int^s_{\eta(s)} \lambda_k^N(Z^N_h(\eta(s)))ds\right)\zeta_{kj}^N\  \bigg|\ Z^N_h(\eta(s))\right]\right)\\
       &=\sum_k (s-\eta(s)) N^{\gamma}N^{c_k}\E\left[ \lambda_k^N(Z^N_h(\eta(s)))\right] (\zeta_{kj}^N)^2\\
       &\hspace{.2in}+\textsf{Var}\left(\sum_k (s-\eta(s))N^{\gamma}N^{c_k} \lambda_k^N(Z^N_h(\eta(s)))\zeta^N_{kj} \right)\\
       &\leq C
       N^{-\rho}N^{\gamma}h+h^2N^{2\gamma} K\sum_k\textsf{Var}\left(\lambda_k^N(Z^N_h(\eta(s)))\right)\\
       &\le C
       N^{-\rho}N^{\gamma}h+h^2N^{2\gamma}K\sum_k\E \left[\left|\lambda_k^N(Z^N_h(\eta(s))-\lambda_k^N(z^N_h(\eta(s)))\right|^2\right]\\
       &\le CN^{-\rho}N^{\gamma}h+dN^{2\gamma} K^2C^2 \left(C_1e^{C_2\cdot (TN^{\gamma})^2}\right)\left(TN^{\gamma}N^{-\rho}\right)h^2,
       \end{split}
        \end{align}
        where we recall that $K$ is the number of reaction channels, and Lemma \ref{lem:main_var_bound_Z} was utilized in the final inequality.

        Turning to $\int_0^1 \nabla F_i^N(Z_h^N(\eta(s)) + r(Z_h^N(s)-Z_h^N(\eta(s))))dr$, we  apply Lemma \ref{lem:orderv} in the appendix with  
  $X_1(s)=Z_h^N(s)$, $X_2(s)=Z^N_h(\eta(s))$, $x_1(s)=z^N_h(s)$, $x_2(s)=z^N_h(\eta(s))$ and $u(x)=[\nabla F_i^N(x)]_j$
          to obtain
           \begin{align}
           \textsf{Var} &\left( \int_0^1 [\nabla F_i^N(Z^N_h(\eta(s)) + r(Z^N_h(s)-Z^N_h(\eta(s)))]_jdr \right)\le  N^{2\gamma} \bar{C}N^{-\rho} \notag,
        \end{align}
        where
        \begin{align}
        \label{eq:Cbar}
        \bar{C}= dC^2 C_1 e^{2C_2 \cdot (TN^\gamma)^2} (TN^{\gamma} N^{-\rho}).
        \end{align}
        In order to apply Lemma \ref{lem:Var_bound} with
        $$A^{N,h}=\int_0^1 [\nabla F_i^N(Z^N_h(\eta(s)) + r(Z^N_h(s)-Z^N_h(\eta(s))))]_jdr$$
   and $B^N =Z_{h,j}^N(s) - Z_{h,j}^N(\eta(s))$, we  note $\left |[\nabla F_i^N]_j \right|\le CN^\gamma$ from our running assumptions to see
   \begin{align}
   &\textsf{Var}\left( \int_0^1 [\nabla F_i^N(Z^N_h(\eta(s)) + r(Z^N_h(s)-Z^N_h(\eta(s))))]_jdr\cdot (Z_{h,j}^N(s)- Z_{h,j}^N(\eta(s)))\right)\notag\\
   &\hspace{.3in} \leq 3C^2\bar{C}N^{2\gamma}N^{-\rho}(N^{\gamma}h)^2  +15C^2N^{2\gamma}\textsf{Var}( Z_{h,j}^N(s) - Z_{h,j}^N(\eta(s)))\notag\\
   &\hspace{.3in}\leq  \hat{C}_1N^{2\gamma}N^{-\rho}(N^{\gamma}h)^2+\hat{C}_2N^{2\gamma}N^{-\rho}N^{\gamma}h,
   \end{align}
   where the final inequality follows from \eqref{eq:24089572} with
   \begin{align}
   \label{eq:bc}
   \begin{split}
   \hat{C}_1&=3C^2 \bar{C} + 15dK^2 C^4 C_1 e^{C_2\cdot (TN^\gamma)^2} TN^\gamma,\\
   \hat{C}_2 &= 15 C^3.
   \end{split}
   \end{align}

   Returning to \eqref{eq:crazy23423}, the above allows us to conclude
   \begin{align*}
   \textsf{Var}(F_i^N(Z^N_h(s))& - F_i^N(Z^N_h(\eta(s)))\\
   &\le d^2\hat{C}_1N^{2\gamma}N^{-\rho}(N^{\gamma}h)^2+d^2\hat{C}_2N^{2\gamma}N^{-\rho}N^{\gamma}h.
   \end{align*}
   Finally, by Lemma \ref{lem:Var_int_bound} in the Appendix,
   \begin{align*}
   \textsf{Var} &\left(\int_0^t F_i^N(Z^N_h(s))- F_i^N(Z^N_h(\eta(s)) ds\right)\\
   &\le d^2T^2\hat{C}_1N^{2\gamma}N^{-\rho}(N^{\gamma}h)^2+d^2T^2\hat{C}_2N^{2\gamma}N^{-\rho}N^{\gamma}h, 
   \end{align*}
   as desired.
   \end{proof}

		We now turn to the middle term of \eqref{eq:3terms}. We first  write
		\begin{align*}
			F_i^N(X^N(s)) -& F_i^N(Z^N_h(s)) = DF_i^N(s)\cdot (X^N(s) - Z^N_h(s)) ,
		\end{align*}
		where
		\[
			DF_i^N(s) = \int_0^1  \nabla F_i^N(Z^N_h(s) + r( X^N(s)-Z^N_h(s))) dr.
		\]
		We will again apply Lemma \ref{lem:Var_bound} to get the necessary bounds.  Therefore, we let
		 $A^N=[DF_i^N(s)]_j$ and $B^N=X^N_j(s) - Z^N_{h,j}(s)$.
				
        Letting $X_1(s)=X^N(s)$, $X_2(s)=Z^N(s)$, $x_1(s)=x^N(s)$, $x_2(s)=z^N_h(s)$ and $u(x)=[\nabla F_i^N(x)]_j$ for an application of Lemma \ref{lem:orderv},  we have
        \begin{align}
        \textsf{Var}(A^N)
            &\leq  N^{2\gamma}\bar{C}N^{-\rho}\notag,
        \end{align}
        where $\bar C$ is defined in \eqref{eq:Cbar} and
        where we use our running assumption that $[\nabla F_i]_j$ is uniformly bounded by $CN^{\gamma}$.
		From \cite[Lemma~3]{AndHigh2012} we know
        \begin{equation*}
     	\E\left[ |B^N | \right] \leq  \tilde{c}_1(e^{\tilde{c}_2N^{\gamma}T}-1)N^{\gamma}h,
        \end{equation*}
        where $\tilde{c}_1$ and $\tilde{c}_2$ are constants independent of $N,\gamma$ and $T$.
 Hence, applying Lemma \ref{lem:Var_bound} we see
        \begin{align*}
          \textsf{Var}&([DF_i^N(s)]_j(X^N_j(s) - Z_{h,j}^N(s)))\\
          &\leq 3N^{2\gamma}\bar{C}\tilde{c}_1^2(e^{\tilde{c}_2N^{\gamma}T}-1)^2N^{-\rho}(N^{\gamma}h)^2+15C^2N^{2\gamma}\textsf{Var}(X^N_j(s) - Z_{h,j}^N(s)),
        \end{align*}
        and
		\[
			\textsf{Var}(F_i^N(X^N(s)) - F_i^N(Z_h^N(s))) \le 15C^2 dN^{2\gamma}\sum_{j=1}^d \textsf{Var}(X^N_j(s) - Z_{h,j}^N(s))+d^2\bar{c}_1 N^{2\gamma} N^{-\rho}(N^{\gamma}h)^2,
		\]
		where
        \begin{align}
        \label{eq:lc}
        \bar{c}_1= 3\bar C \tilde c_1^2(e^{\tilde c_2 TN^\gamma}- 1)^2.
        \end{align}
		
		Finally returning to \eqref{eq:3terms}, we may combine all of the above to see
		\begin{align*}
			&\textsf{Var}(X^N_i(t)- Z_{h,i}^N(t)) \le 3 c_1Te^{c_2N^{\gamma}T}N^{2\gamma}(N^{-\rho}h) \\
			&\hspace{.2in}+ 
			3 \left[ 15C^2 dN^{2\gamma}t\sum_{j=1}^d \int_0^t \textsf{Var}(X^N_j(s) - Z_{h,j}^N(s))\; ds +d^2 \bar{c}_1 N^{2\gamma} N^{-\rho}(N^{\gamma}h)^2T^2\right]\tag{Lemma \ref{lem:Var_int_bound}}\\
			&\hspace{.2in} + 3\left[d^2\hat{C}_1 \cdot (TN^{\gamma})^2 N^{-\rho}(N^{\gamma}h)^2+d^2 \hat{C}_2 \cdot (T N^{\gamma})^2N^{-\rho}N^{\gamma}h\right]\tag{Lemma \ref{lemma:third_term}}
			\\[1ex]
&\leq 3d^2T^2(\hat{C}_1+\bar{c}_1)N^{2\gamma}N^{-\rho}(N^{\gamma}h)^2+3(d^2T^2\hat{C}_2N^{2\gamma}+c_1N^{\gamma}Te^{c_2N^{\gamma}T})N^{-\rho}N^{\gamma}h\\
&\quad + 45 C^2TdN^{2\gamma}\sum_{j=1}^{d} \int_0^t \textsf{Var}(X^N_j(s) - Z_{h,j}^N(s)) ds.
		\end{align*}
		Thus, setting
		\[
			g(t) = \max_{i\in \{1,\dots,d\}} \{ \textsf{Var}([X^N(t)]_i - [Z_h^N(t)]_i)\},
		\]
		we have
		\begin{align*}
			g(t) &\leq 3d^2T^2(\hat{C}_1+\bar{c}_1)N^{2\gamma}N^{-\rho}(N^{\gamma}h)^2+3(d^2T^2\hat{C}_2N^{2\gamma}+c_1N^{\gamma}Te^{c_2N^{\gamma}T})N^{-\rho}N^{\gamma}h\\
&\quad + 45 C^2Td^2N^{2\gamma} \int_0^t g(s) ds,
		\end{align*}
		and the result under the assumption that $f_i(x) = x_i$ is shown by an application of Gronwall's inequality,
		\begin{align}
          	\textsf{Var}(X^N_i(s) - Z_{h,i}^N(s))      &\leq 3e^{45 C^2T^2d^2N^{2\gamma}}d^2T^2(\hat{C}_1+\bar{c}_1)N^{2\gamma}N^{-\rho}(N^{\gamma}h)^2\notag\\
	&\hspace{.2in}+ 3e^{45 C^2T^2d^2N^{2\gamma}}(d^2T^2\hat{C}_2N^{2\gamma}+c_1N^{\gamma}Te^{c_2N^{\gamma}T})N^{-\rho}N^{\gamma}h\notag\\
	&= \tilde{C}_1 \cdot N^{-\rho} (N^\gamma h)^2 + \tilde{C}_2 \cdot N^{-\rho} N^\gamma h,\label{eq:2409857}
        \end{align}
        where $\tilde{C}_1$ and $\tilde{C}_2$ are defined by the above equality.
		
		To show the general case, note that from Lemma \ref{lem:taylor} in the appendix we have
		\begin{align*}
			f(X^N(t)) - f(Z^N_h(t)) &= \int_0^1 \nabla f(Z^N_h(t) + r(X^N(t)-Z^N_h(t)))dr\cdot (X^N(t)-Z^N_h(t))\notag.
		\end{align*}
          We let $X_1(t)=X^N(t)$, $X_2(t)=Z^N_h(t)$, $x_1(t)=x^N(t)$, $x_2(t)=z^N_h(t)$ and $u(x)= [\nabla f(x)]_j$ for an application of Lemma \ref{lem:orderv}, which yields
          \begin{align}
        &\textsf{Var}\left(\int_0^1 [\nabla f(Z^N_h(t) + r(X^N(t)-Z^N(t)))]_jdr\right) \leq  dM^2 \bar{C}N^{-\rho},
        \end{align}
        where $M$ is the uniform bound of $\nabla_{ij} f(x)$ and the factor of $d$ appears in the application of Lemma \ref{lem:orderv} since $|[\nabla f(x)]_j - [\nabla f(y)]_j| \le M \sqrt{d} |x-y|$ for each $j$.
 Hence, by an application of Lemma \ref{lem:Var_bound} and the work above we see,
        \begin{align}
          \textsf{Var}&\left(\int_0^1 \nabla_j f(Z^N_h(t) + r(X^N(t)-Z^N_h(t)))dr\cdot(X^N_j(s) - Z_{h,j}^N(s))\right)\notag\\
          &\leq  dM^2 \bar{c}_1N^{-\rho} (N^{\gamma}h)^2+15M^2\textsf{Var}(X^N_j(s) - Z_{h,j}^N(s))\notag.
        \end{align}
  Thus, utilizing \eqref{eq:2409857}, we have
        \[
          \textsf{Var}(f(X^N(t)) - f(Z^N_h(t)) ) \leq \bar{C_1} \cdot N^{-\rho}(N^{\gamma}h)^2+\bar{C_2}\cdot N^{-\rho}N^{\gamma}h,
        \]
        where
        \begin{align}
        \label{eq:C1}
        \bar{C_1}= 15d^2M^2\tilde C_1+d^3m^2 \bar{c}_1,
        \end{align}
        and
         \begin{align}
         \label{eq:C2}
        \bar{C_2}=15d^2M^2 \tilde C_2.
        \end{align}
        Note that  the functional dependence of  $\bar C_1$ and $\bar C_2$ on $N$ and $T$ is via the product $N^\gamma T$.
	\end{proof}

\section{Consequences for Complexity}
\label{sec:complexity}

In the present section we assume that $\gamma \le 0$ (which holds, for example, in the classical scaling).  We further assume that $\bar{C_1}$ is not significantly larger than $\bar{C_2}$, where $\bar{C_1}$ and $\bar{C_2}$ are the constants presented in Theorem~\ref{thm:var}.   As detailed in the discussion following Theorem~\ref{thm:var}, under these assumptions our upper bounds on the variances given  in Theorem \ref{thm:var} and Corollary~\ref{thm:cor} are well approximated
by constants multiplied by $N^{-\rho} (N^{\gamma} h)$.   
This scaling is verified numerically on an example in Section \ref{sec:comp}.   Theoretically demonstrating the validity of the assumption pertaining to the relative sizes of the constants will require a finer analysis  than is carried out here.


Our theoretical results, combined with the assumptions outlined above, imply
that the variances in the level-wise estimators
$
\Ql_{\ell}
$
and
$\Ql_E$
appearing in
(\ref{eq:QL_real})
and
(\ref{ub3})
satisfy
\begin{equation}
\textsf{Var}( \Ql_{\ell} ) \le C\cdot \frac{N^{-\rho} N^\gamma h_{\ell}} {n_{\ell}}
   \qquad
 \mathrm{and}
  \qquad
 \textsf{Var}( \Ql_E ) \le C\cdot \frac{N^{-\rho} N^\gamma h_{L}} {n_{E}},
\label{eq:vqle}
   \end{equation}
where  $C$ is a generic constant.
Compared with the original analysis in
\cite[Theorems~1~and~2]{AndHigh2012}, the important refinement is deletion of the $h_\ell^2$ terms found in the similar bounds of \cite{AndHigh2012}.  These $h_\ell^2$ terms
dominate in most cases, including when $h_\ell > N^{-1}$ and the system satisfies the classical scaling (recall that $\rho = 1$ in the classical scaling).

Further,
in the basic inequality
\[
\textsf{Var}( f(Z^N_{\ell_{0}}(t)) ) \le \left(   \sqrt{ \textsf{Var}( f(Z^N_{\ell_{0}}(t)) -
  f( X^N(t))) } +
       \sqrt{ \textsf{Var}(f( X^N(t))) }  \right)^2
\]
we may use
Theorem~\ref{thm:var}
with $h = h_{\ell_{0}}$ 
to control
$ \textsf{Var}( f(Z_{\ell_{0}}) - f(X^N)) ) $
and
Lemma~\ref{lem:main_var_bound_X}
to control
$\textsf{Var}( f(X^N) )$, to find that
\begin{equation}
\textsf{Var}( \Ql_{0} ) \le C\cdot \frac{N^{-\rho} N^\gamma} {n_{0}}.
\label{eq:vql}
\end{equation}
Taking a variance in
(\ref{eq:estimator_real})
it then follows from
 (\ref{eq:vqle})
and
(\ref{eq:vql})
that to leading order
\[
\textsf{Var}(
\Ql_{B}
)
 \le C  N^{-\rho} N^\gamma
  \left( \frac{1}{{n_{0}}}   +
 \sum_{\ell = \ell_0+1}^L
   \frac{h_{\ell}}{n_{\ell}} \right).
\]

To achieve the required overall variance of $\epsilon^2$, and to spread the variance budget fairly evenly across the levels,
 we may then use, to order of magnitude,
\begin{equation}
    n_0 =  N^{-\rho} N^\gamma (L-\ell_0+1) \epsilon^{-2}
\qquad
 \mathrm{and}
\qquad
 n_{\ell} =  N^{-\rho} N^\gamma (L - \ell_0+1) h_{\ell} \epsilon^{-2}.
\label{eq:nvals}
\end{equation}
For the biased estimator
(\ref{eq:estimator_real})
we take $h_L = O(\epsilon)$ in order for the discretization error to be
within our target accuracy.
This gives
$L = O(|\ln(\epsilon)|)$ levels.
The computational cost of each pair of
tau-leap paths is proportional to $ h_{\ell}^{-1}$, and hence the overall complexity is
of magnitude
\begin{equation}
 n_0  h_{\ell_0}^{-1} +
            \sum_{\ell = \ell_0+1}^L
                             n_\ell h_{\ell}^{-1}
  =
   N^{-\rho} N^\gamma   \epsilon^{-2} \left(  (L-\ell_0 + 1) h_{\ell_0}^{-1} + (L-\ell_0+1)(L-\ell_0) \right).
\label{eq:mag1}
\end{equation}
Based on this analysis, which we recall is performed under the assumption that $\gamma \le 0$, we take 
$\ell_0 = 0$.
Doing so yields a computational complexity of leading order
\[
	N^{-\rho}N^\gamma \epsilon^{-2} \ln(\epsilon)^2.
\]

Still assuming that $\ell_0 = 0$, for the unbiased estimator
(\ref{eq:unbiased_MLMC})
we may take, to leading order of magnitude,
\begin{equation*}
    n_0 =  N^{-\rho} N^\gamma (L+2) \epsilon^{-2},
\ \
\ \ \ 
 n_{\ell} =  N^{-\rho} N^\gamma (L +2) h_{\ell} \epsilon^{-2}, \ \ \
 \mathrm{and}
\ \ \  
n_{E} =  N^{-\rho} N^\gamma (L+2)  h_{L} \epsilon^{-2}.
\end{equation*}
If we again use $h_L = O(\epsilon)$, then the computational complexity for the unbiased method
is bounded in magnitude by
\begin{equation}
   N^{-\rho} N^\gamma   \epsilon^{-2} \ln(\epsilon)^2  + \max\{N^{-\rho}N^{\gamma}|\log(\epsilon)|\cdot h_L \epsilon^{-2} , 1\}\cdot \overline N,
\label{eq:mag2}
\end{equation}
  where we recall
that
$\overline{N }$, defined in
(\ref{eq:Nbar}),
is the cost of computing a sample path with the exact method, and we note that the maximum appears because some exact paths are necessarily computed for the unbiased method.  We also note that we are implicitly assuming that $h_L^{-1}$ is not appreciably larger than $\overline N$, which we believe is reasonable.   Finally, while we computed the above under the assumption that $h_L  = O(\epsilon)$, we kept $h_L \epsilon^{-2}$ in the second term of \eqref{eq:mag2}  instead of simply writing $\epsilon^{-1}$ in order to explicitly point out the dependence on each term.


We note that the complexity bounds derived in \cite{AndHigh2012}, which considered only the $L^2$ norm of the difference of the coupled processes,  have another term  of order $\max\{N^{2\gamma} h_L^2 \epsilon^{-2}, 1\}\cdot \overline N$ added to \eqref{eq:mag2}.   This term was often the dominating one, and has been removed by the direct analysis on the variance presented here.

To finish this section we point out that the
analysis produces upper bounds on the computational complexity---in particular, the choices
for the number of samples paths per level 
 are sufficient to achieve the required accuracy, based on bounds on the individual variances
and with the strategy of spreading the cost evenly between levels, but we have not shown that they are optimal.
In practice, and as described more fully in
\cite{AndHigh2012},
for a given problem,  and
with a small amount of further computation,
it is possible to perform an initial optimization in order to choose these key parameters
adaptively. Hence, practical performance may outstrip these complexity bounds.

\section{Computational Test}
\label{sec:comp}

The efficiency of multilevel Monte Carlo tau-leaping was demonstrated computationally 
in \cite{AndHigh2012}, so we restrict ourselves here to 
 testing the sharpness of our new analytical results
on a simple non-linear model.  We consider 
a particular case from the family of models presented in Example~\ref{example:running1}.

\begin{example}

Consider 
the case where 
$M = 0.2$ in   the model of
 Example~\ref{example:running1}, defined through 
 \eqref{eq:coupled_example}.
 \begin{figure}[t]
\centering
	\subfigure[Varying $N$ with $h=0.001$ fixed.]{\includegraphics[width=0.47\textwidth]{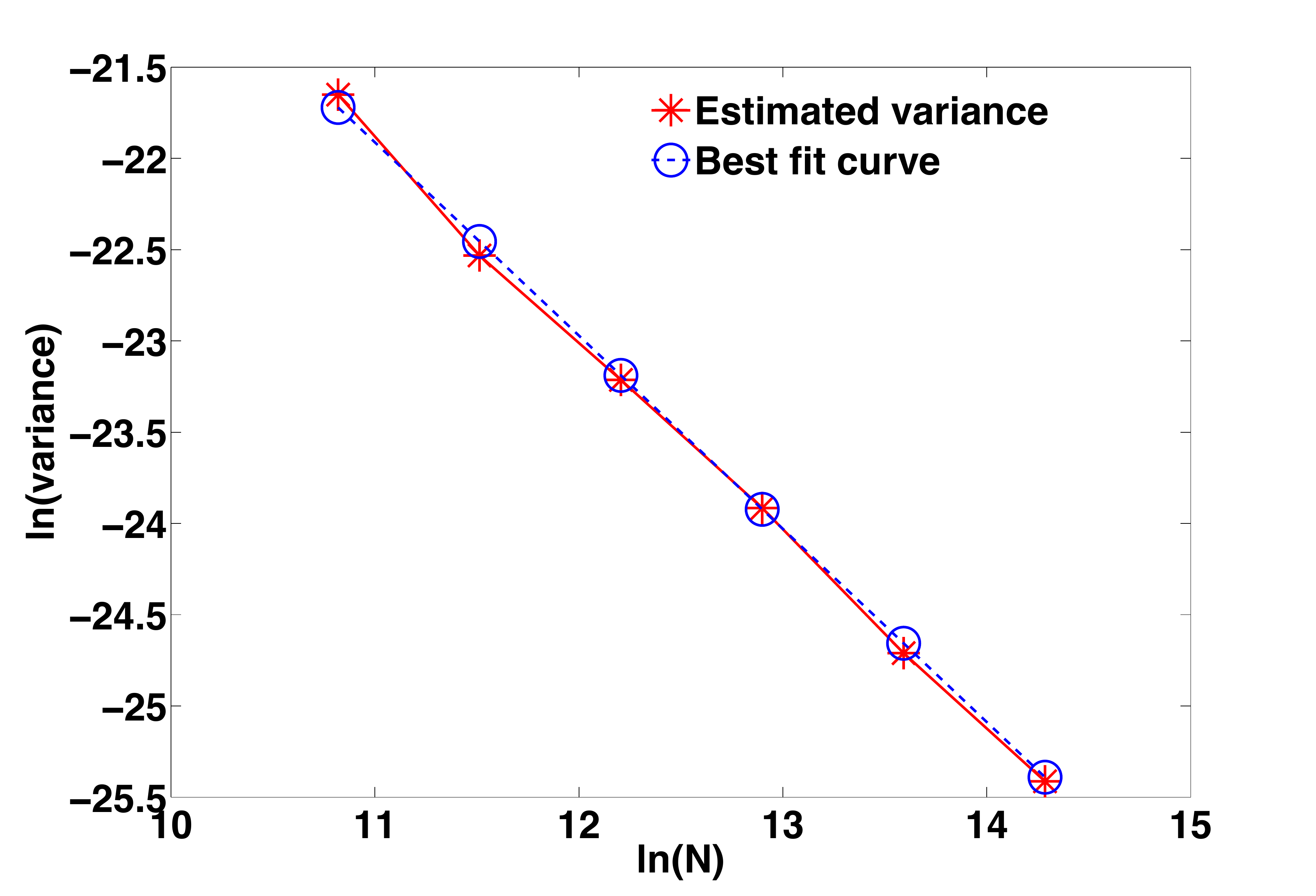}}\qquad
	\subfigure[Varying $h$ with $N = 10^6$ fixed.]{\includegraphics[width=0.47\textwidth]{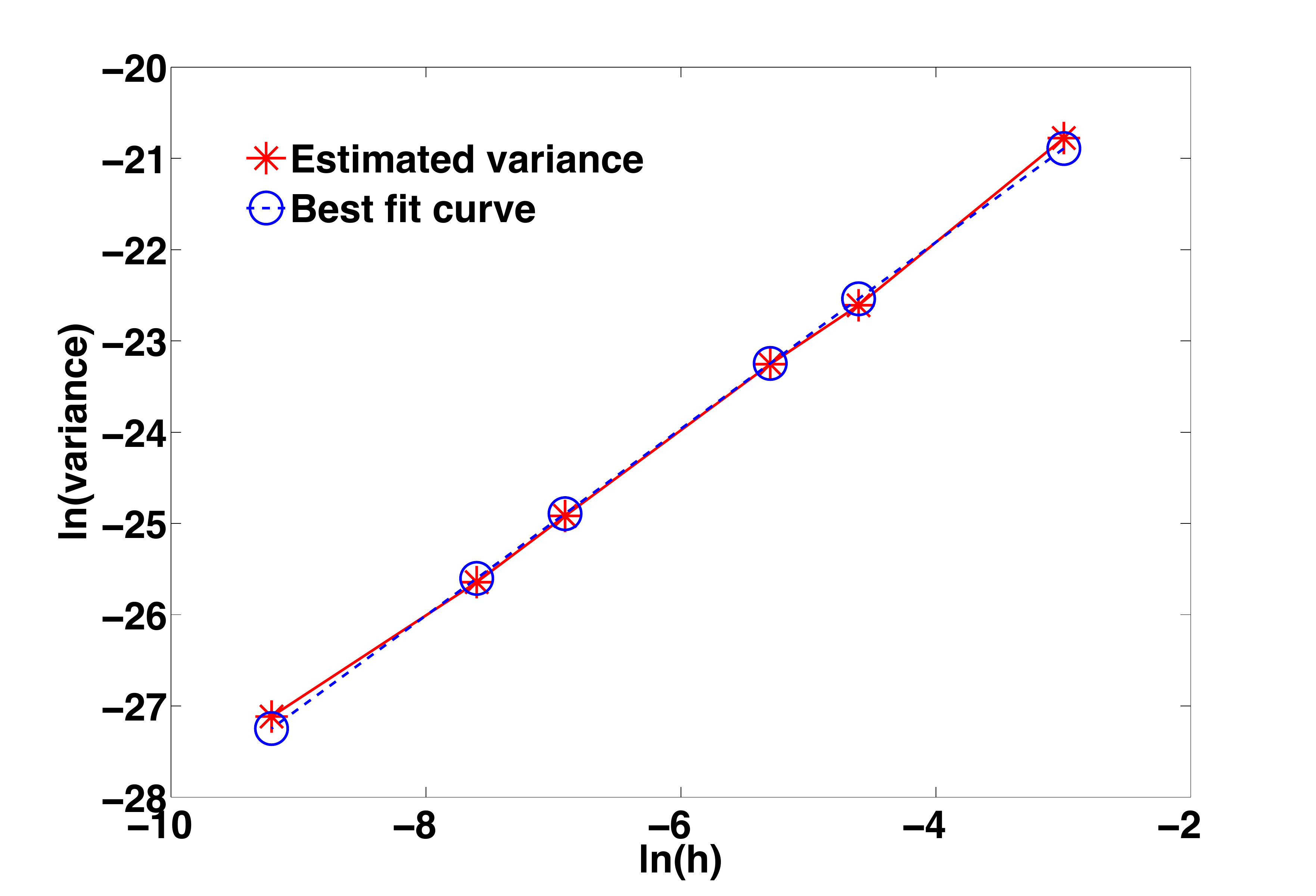}}
	\caption{Log-log plots of $\textsf{Var}(X_1^N(t) - Z_{h,1}^N(t))$ for the model in Example \ref{example:simple}.  In (a), $h$ is held constant while $N$ is changed.  In (b), $N$ is fixed while $h$ is varied.  The best fit curve for all the data is overlain in the dashed blue line.}
	\label{fig:example_simple}
\end{figure}
In Figure~\ref{fig:example_simple}, we provide log-log plots of $\textsf{Var}(X^N_1(T) - Z^N_{h,1}(T))$ for the coupled processes with $T = 0.3$, and varying values for $N$ and $h$.  The plots are consistent with the functional form
\begin{equation}\label{eq:functional_form}
	\textsf{Var}(X^N_1(t) - Z^N_{h,1}(t)) \approx C N^{-1} h,
\end{equation}
matching the bound arising from 
Theorem~\ref{thm:var}.
The best fit curve for the data, obtained by a least squares approximation and which is shown in each image, is $\textsf{Var}(X^N_1(t) - Z^N_{h,1}(t)) \approx 0.0408 \cdot N^{-1.0588} h^{1.0228}$.

In Figure \ref{fig:example_simple_TT}, we provide log-log plots of $\textsf{Var}(Z^N_{\ell,1}(T) - Z^N_{\ell -1,1}(T))$ for the coupled processes with $T = 0.3$, and varying values for $N$ and $h_\ell$.  These plots also follow the functional form of \eqref{eq:functional_form}, 
matching the bound arising from 
Corollary~\ref{thm:cor},
with best fit curve of  $\textsf{Var}(Z^N_{\ell,1}(T) - Z^N_{\ell -1,1}(T))\approx 0.1038 \cdot N^{-1.0279} h_\ell^{0.9845}$.

\begin{figure}
\centering
	\subfigure[Varying $N$ with $h=0.001$ fixed.]{\includegraphics[width=0.47\textwidth]{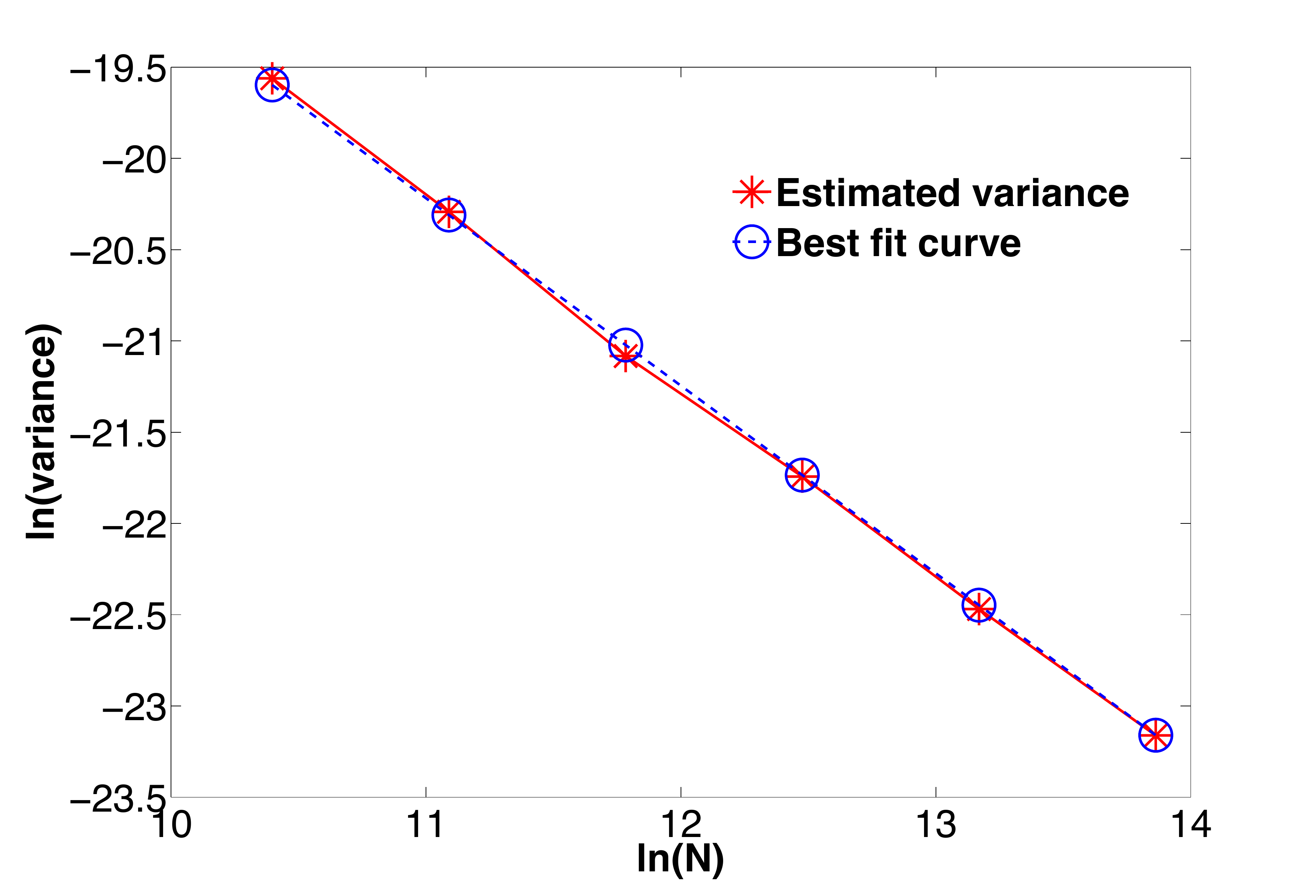}}\qquad
	\subfigure[Varying $h$ with $N = 10^6$ fixed.]{\includegraphics[width=0.47\textwidth]{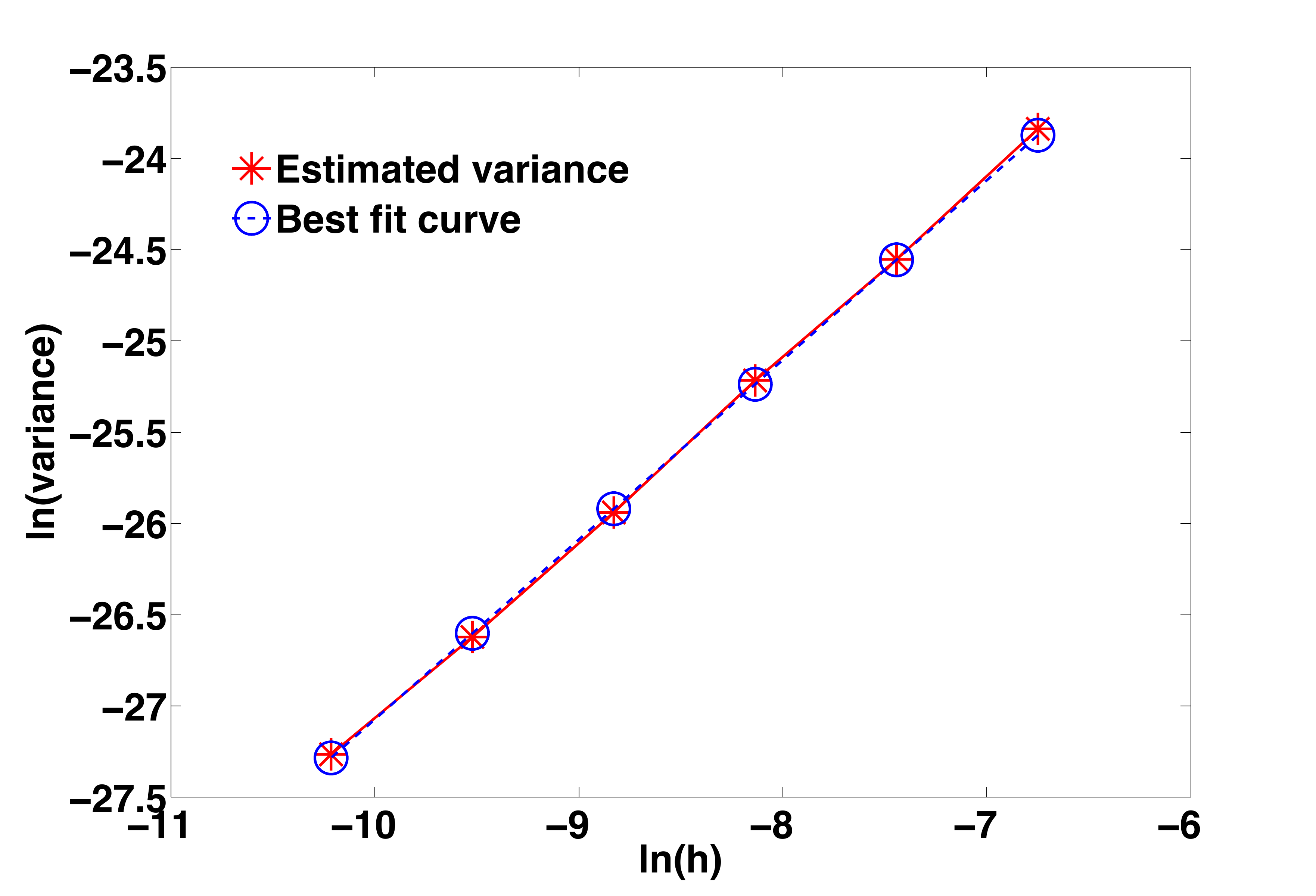}}
	\caption{Log-log plots of $\textsf{Var}(Z_{\ell,1}^N(t) - Z_{\ell-1,1}^N(t))$ for the model in Example \ref{example:simple}.  In (a), $h$ is held constant while $N$ is changed.  In (b), $N$ is fixed while $h$ is varied.  The best fit curve for all the data is overlain in the dashed blue line.}
	\label{fig:example_simple_TT}
\end{figure}

\label{example:simple}
\end{example}
 \section{Conclusions}
\label{sec:conc}

The main contribution of this work is to add further theoretical support
for multilevel Monte Carlo tau-leaping by developing new
complexity bounds that behave well
for large values of the system size parameter.
To do this, we took the novel step of directly
 estimating the variance between
 pairs of  paths, rather than proceeding via a mean-square
 convergence property.
We also provided numerical support showing our estimates for the variances are sharp.

Stochastic simulation of continuous time, discrete space, Markov chains
is a bottleneck across a range of application areas, and
there are many promising directions
for further study of multilevel Monte Carlo in this context. In particular,
specific instances of the very general
scaling considered here could be
used in order to develop more customized
strategies, and complexity bounds,
in suitable model classes; for example, where there is a
 known separation of scale.
 
 The analysis presented is valid for $\gamma \le 0$.  For the case $\gamma>0$, which is the regime of ``stiff'' systems, it is often possible to generate, via averaging techniques, a reduced model satisfying $\gamma \le 0$.  Taking this reduced model as the ``finest level'' in a MLMC framework is then a natural way to proceed in the construction of an efficient Monte Carlo method.  This procedure was carried out successfully in Section 9  of \cite{AndHigh2012} for an example of viral growth.

\vspace{.225in}

\noindent \textbf{\large Acknowledgements.}  We thank two anonymous reviewers for detailed comments on this work.  DFA was supported by NSF grants DMS-1009275 and DMS-1318832 and Army Research Office grant W911NF-14-1-0401.  DJH was supported by a Royal Society/Wolfson Research Merit Award and
a Royal Society/Leverhulme Trust Senior Fellowship.  YS was supported by NSF-DMS-1318832.

\appendix

\section{Technical Details from the Analysis}
\label{sec:app}

We provide here some technical lemmas which were used in Section \ref{sec:res}.

\begin{lemma}\label{lem:twovar_bound}
 \label{lem:orderv}
Suppose $X_1(s)$ and $X_2(s)$ are  stochastic processes on $\R^d$ and that   $x_1(s)$ and $x_2(s)$ are  deterministic processes on $\R^d$.  Further, suppose that
 \begin{align}
 \label{eq:orderv}
        \sup_{s \le T} \E\left[ |X_1(s)-x_1(s)|^2\right] \le \widehat{C}_1(N^{\gamma}T)N^{-\rho},\quad  \sup_{s \le T} \E \left[ |X_2(s)-x_2(s)|^2\right] \le \widehat{C}_2(N^{\gamma}T)N^{-\rho},
 \end{align}
 for some $\widehat{C}_1,\widehat{C}_2$ depending upon $N^{\gamma}T$.
   Assume that $u: \R^d \to \R$ is Lipschitz  with Lipschitz constant $L$. Then,
		 \[
         \sup_{s \le T}\textsf{Var} \left( \int_0^1  u(X_2(s) + r(X_1(s)-X_2(s)))dr\right)\le L^2 \max(\widehat{C}_1,\widehat{C}_2)N^{-\rho}.
		 \]
 \end{lemma}
  \begin{proof}
  First we know,
          \begin{align*}
        \textsf{Var} &\left( \int_0^1  u(X_2(s) + r(X_1(s)-X_2(s)))dr \right)\notag\\
		&= \textsf{Var} \left(\int_0^1   u(X_2(s) + r(X_1(s)-X_2(s)))- u(x_2(s) + r(x_1(s)-x_2(s))) dr\right)\notag\\
            &\leq \E \left(\int_0^1    u(X_2(s) + r(X_1(s)-X_2(s)))- u(x_2(s) + r(x_1(s)-x_2(s))) dr\right)^2\notag\\
            &\leq \int_0^1 \E [ u(X_2(s) + r(X_1(s)-X_2(s)))- u(x_2(s) + r(x_1(s)-x_2(s)))]^2 dr.
            \end{align*}
           Using that $u$ is Lipschitz, we  may continue
           \begin{align*}
           \textsf{Var} &\left( \int_0^1  u(X_2(s) + r(X_1(s)-X_2(s)))dr \right)\notag\\
            &\leq L^2\int_0^1 \E\left[ |X_2(s) + r(X_1(s)-X_2(s))-(x_2(s) + r(x_1(s)-x_2(s)))|^2\right] dr\notag\\
            &= L^2\int_0^1 \E \left[ |r(X_1(s)-x_1(s)) + (1-r)(X_2(s)-x_2(s))|^2 \right] dr\notag\\
            &\leq L^2\int_0^1 r \E\left[ |X_1(s)-x_1(s)|^2\right] + (1- r)\E \left[| X_2(s)-x_2(s)|^2 \right] dr\notag\\
            &\leq L^2\max(\widehat{C}_1,\widehat{C}_2)N^{-\rho},
        \end{align*}
        where the second to last inequality follows from convexity of the quadratic function, and the final inequality holds from applying \eqref{eq:orderv}.
  \end{proof}

\begin{lemma}\label{lem:Var_bound}
	Suppose that $A^{N,h}$ and $B^{N,h}$ are families of random variables determined by  scaling parameters $N$ and $h$.  Further, suppose that there are $C_1>0, C_2>0$ and $C_3>0$ such that  for all $N>0$ the following three conditions hold:
	\begin{enumerate}
		\item $\textsf{Var}(A^{N,h}) \le C_1 N^{-\rho}$ uniformly in $h$.
		\item $|A^{N,h}| \le C_2N^\gamma$ uniformly in $h$.
		\item $|\E [B^{N,h}]| \le C_3N^{\gamma} h$.
	\end{enumerate}
	 Then
	\begin{equation*}
		\textsf{Var}(A^{N,h} B^{N,h}) \le 3C_3^2C_1N^{-\rho}(N^{\gamma}h)^2+15C_2^2N^{2\gamma}\textsf{Var}(B^{N,h}).
	\end{equation*}
\end{lemma}
 \begin{proof}
Via a direct expansion, the variance of the product can be represented in the following manner
 \begin{align*}
 \textsf{Var}(A^{N,h} B^{N,h}) &= \E[ (\E [B^{N,h}])(A^{N,h}-\E [ A^{N,h}])+(\E [A^{N,h}])(B^N-\E [B^{N,h}])\\
 &\hspace{.2in}
 +(A^{N,h}-\E[ A^{N,h}])(B^{N,h}-\E [B^{N,h}]) -\E[(A^{N,h}-\E[ A^{N,h}])(B^{N,h}-\E [B^{N,h}])]]^2.
  \end{align*}
Using the basic bound $(a + b + c)^2 \le 3a^2 + 3 b^2 + 3 c^2$, we have
\begin{align*}
	\textsf{Var}(A^{N,h} B^{N,h}) \le&  3(\E [B^{N,h}])^2\textsf{Var}(A^{N,h}) +  3(\E[ A^{N,h}])^2\textsf{Var}(B^{N,h})\\
 &\hspace{.1in}+3\textsf{Var}((A^{N,h}-\E [A^{N,h}])(B^{N,h}-\E [B^{N,h}])).
\end{align*}
Using our assumptions in the statement of the lemma, the following two inequalities are immediate
\begin{align}\label{eq:appendix_inequ1}
	3(\E[ B^{N,h}])^2\textsf{Var}(A^{N,h}) &\le 3 C_3^2 C_1 N^{-\rho} (N^\gamma h)^2,
\end{align}
and
\begin{align}\label{eq:appendix_inequ2}
	3(\E [A^{N,h}])^2\textsf{Var}(B^{N,h}) &\le 3C_2^2 N^{2\gamma} \textsf{Var}(B^{N,h}).
\end{align}
For the final term we  bound the variance by the second moment to achieve
\begin{align}\label{eq:appendix_inequ3}
\begin{split}
	3\textsf{Var}((A^{N,h}-\E [A^{N,h}])(B^{N,h}-\E [B^{N,h}])) &\le 3\E((A^{N,h}-\E [A^{N,h}])(B^{N,h}-\E [B^{N,h}]))^2\\
	&\le 12 C_2^2 N^{2\gamma} \textsf{Var}(B^{N,h}).
\end{split}
\end{align}
Combining \eqref{eq:appendix_inequ1}, \eqref{eq:appendix_inequ2}, and \eqref{eq:appendix_inequ3} gives the desired result.
%
 \end{proof}

\begin{lemma}\label{lem:Var_int_bound}
Let $Q(s)$ be a stochastic process for which $\sup_{s \in [a,b]} \textsf{Var(Q(s))} < \infty$.  Then
\begin{align*}
\begin{split}
	\textsf{Var}\left( \int_a^b Q(s) ds\right)	 &\le (b-a) \int_a^b \textsf{Var}(Q(s)) ds.
\end{split}	
\end{align*}

\end{lemma}

\begin{proof}
The proof is straightforward.
\begin{align*}
	\textsf{Var}\left( \int_a^b Q(s) ds\right) &= \E \left( \int_a^b Q(s) ds - \E \left[\int_a^b Q(s) ds\right]\right)^2 = \E \left( \int_a^b (Q(s) - \E [Q(s)]) ds \right)^2\\
	&\le (b-a) \int_a^b \E\left[  \left( Q(s) - \E Q(s)\right)^2\right] ds=(b-a) \int_a^b \textsf{Var}(Q(s))ds.\qedhere
\end{align*}
\end{proof}

 \begin{lemma}\label{lem:taylor}
 	Let $f:\R^d \to \R$ have continuous first derivative.  Then, for any $x,y\in \R^d$,
	\[
		f(x) = f(y) + \int_0^1 \nabla f(sx + (1-s)y) ds \cdot (x-y).
	\]
 \end{lemma}
 \begin{proof}
 	Let $H(t) = f(tx + (1-t)y)$.  Then $H'(t) = \nabla f(tx + (1-t)y)\cdot (x-y),$
	and by the fundamental theorem of calculus, $H(1) = H(0) + \int_0^1 H'(s)ds,$
	which is equivalent to the statement of the lemma.
 \end{proof}

 \bibliographystyle{siam}

\bibliography{MLMC_TAU}

\end{document}